\documentclass[12pt]{amsart}
\usepackage{mathrsfs}

\usepackage{amssymb,amsmath,graphicx,color,textcomp, amsthm,bbm,bbold, enumerate,booktabs}
\usepackage{chngcntr}
\usepackage{apptools}
\AtAppendix{\counterwithin{theorem}{section}}
\definecolor{Red}{cmyk}{0,1,1,0}

\definecolor{verde}{cmyk}{1,0,1,0}

\definecolor{loka}{cmyk}{.5,0,1,.5}
\definecolor{azul}{cmyk}{1,1,0,0}

\numberwithin{equation}{section}

\newcommand{\be}{\begin{equation}}
\newcommand{\ee}{\end{equation}}

\newtheorem{theorem}{Theorem}

\newtheorem{definition}{Definition}

\begin{document}
\title{A new truncated $M$-fractional derivative type unifying some fractional derivative types with classical properties}
\author{J. Vanterler da C. Sousa$^1$}
\address{$^1$ Department of Applied Mathematics, Institute of Mathematics,
 Statistics and Scientific Computation, University of Campinas --
UNICAMP, rua S\'ergio Buarque de Holanda 651,
13083--859, Campinas SP, Brazil\newline
e-mail: {\itshape \texttt{ra160908@ime.unicamp.br, capelas@ime.unicamp.br }}}

\author{E. Capelas de Oliveira$^1$}

\begin{abstract} We introduce a truncated $M$-fractional derivative type for $\alpha$-differentiable functions that generalizes four other fractional derivatives types recently introduced by Khalil et al., Katugampola and Sousa et al., the so-called conformable fractional derivative, alternative fractional derivative, generalized alternative fractional derivative and $M$-fractional derivative, respectively. We denote this new differential operator by $_{i}\mathscr{D}_{M}^{\alpha,\beta }$, where the parameter $\alpha$, associated with the order of the derivative is such that $ 0 <\alpha<1 $, $\beta>0$ and $ M $ is the notation to designate that the function to be derived involves the truncated Mittag-Leffler function with one parameter.

The definition of this truncated $M$-fractional derivative type satisfies the properties of the integer-order calculus. We also present, the respective fractional integral from which emerges, as a natural consequence, the result, which can be interpreted as an inverse property. Finally, we obtain the analytical solution of the $M$-fractional heat equation and present a graphical analysis.

\end{abstract}
\maketitle


\section{Introduction} 
The non integer-order calculus or fractional calculus, as it is largely diffused, is as important and ancient as the integer-order calculus, and for many years the scientific community didn't know it. Currently, there are numerous and important definitions of fractional derivatives types, each one of them with its peculiarity and application \cite{RHM,AHMJ,IP}. Although, to be highlighted only from 1974, after the first international conference on fractional calculus, it has shown to be important and with great applicability in modeling problems, more precisely natural phenomena. For this growth, it was necessary the contribution of famous mathematicians, such as Lagrange, Abel, Euler, Liouville, Riemann, as well as recently Caputo and Mainardi, among others.

It is possible to define various integrals and fractional derivatives. Each definition has its own peculiarity and thus makes the fractional calculus fruitful in the sense of theory and applications. We report some types of fractional derivatives that have been introduced so far, among which we mention: Riemann-Liouville, Caputo, Hadamard, Caputo-Hadamard, Riesz, among others \cite{IP,ECJT}. Many of these derivatives are defined from the fractional integral in the Riemann-Liouville sense. Recently, Katugampola \cite{UNT}, has presented a new fractional integral unifying six existing fractional integrals, named Riemann-Liouville, Hadamard, Erdélyi-Kober, Katugampola, Weyl and Liouville \cite{IP, RCEC}.

On the other hand, also recent, Khalil et al. \cite{RMAM} proposed the so-called compatible fractional derivative of order $\alpha$ with $0<\alpha<1$ in order to generalize the classical properties of calculus. More recently, in 2014, Katugampola \cite{UNK} has also proposed an alternative fractional derivative with classical properties, which refers to the Leibniz and Newton calculus, similar to the conformable fractional derivative. In 2017, Sousa and Oliveira \cite{JE}, introduced an $M$-fractional derivative involving a Mittag-Leffler function with one parameter \cite{GKAM} that also satisfies the properties of integer-order calculus. In this sense, we are going to introduce a truncated $M$-fractional derivative type that unifies four existing fractional derivative types mentioned above and which also satisfies the classical properties of integer-order calculus.

The study of differential equations has proved very useful over time. One of the main reasons is that the simplest differential equations have the ability to model more complex natural systems \cite{RHM,IP,kilbas}. There are a larger number of application involving differential equations of which we mention: the problem of population dynamics, brachistochronous problem, wave equation, heat equation and others \cite{RHM,ZE1}. However, natural systems over time, become more complex and more than differential equations, provides a rough and simplified description of the actual process, it is necessary that new and more refined mathematical tools are presented and studied. In this sense, fractional derivatives are used to propose modeling in order to obtain more precise results in the studies and applications involving differential equations \cite{kilbas}. Then through the use of properties of a truncated $M$-fractional derivative type, we will present an analytical study of the heat equation and through graph, we will analyze the behavior of the solution in relation to other types of fractional derivatives the so-called local derivatives.

This paper is organized as follows: in section 2, our main result, we introduce the concept of truncated $M$-fractional derivative type involving a truncated Mittag-Leffler function, as well as several theorems. Also, we introduce the respective $M$-fractional integral for which we demonstrate the inverse property. In section 3, we present the relationship between a truncated $M$-fractional derivative type, introduced here, and the conformable fractional derivative, generalized and alternative fractional derivative and $M$-fractional derivative. In section 4, we perform an analytical study of the $M$-fractional heat equation in order to obtain the analytical solution and present some graphs. Concluding remarks close the paper.


\section{Truncated $M$-fractional derivative type}

In this section, we define a truncated $M$-fractional derivative type and obtain several results that have a great similarity with the results found in the classical calculus. From the definition, we present a theorem showing that this truncated $M$-fractional derivative type is linear, obeys the product rule and the composition of two $\alpha$-differentiable functions, the quotient rule and the chain rule. It is also shown that the derivative of a constant is zero, as well as versions for Rolle's theorem, the mean value theorem and an extension of the mean value theorem. Further, the continuity of this truncated $M$-fractional derivative type is shown as in integer-order calculus. Also, we introduce the concept of $M$-fractional integral of a $f$ function. From the definition, we shown the inverse theorem.

We define the truncated Mittag-Leffler function of one parameter by:
\begin{equation}\label{2A}
_{i}\mathbb{E}_{\beta }\left(z\right) =\overset{i}{\underset{k=0}{\sum }}\frac{z^{k}}{\Gamma \left( \beta k+1\right) },
\end{equation}
with $\beta>0$ and $z\in\mathbb{C}$.

From Eq.(\ref{2A}), we define a truncated $M$-fractional derivative type that unifies other four fractional derivatives that refer to classical properties of the integer-order calculus.

In this work, if a truncated $M$-fractional derivative type of order $\alpha$ as defined in Eq.(\ref{J}) of a function $f$ exists, we say that the function $f$ is $\alpha$-differentiable.

Thus, let us begin with the following definition, which is a generalization of
the usual definition of integer order derivative.

\begin{definition} Let $f:\left[ 0,\infty \right) \rightarrow \mathbb{R}$. For $0<\alpha <1$ a truncated $M$-fractional derivative type of $f$ of order $\alpha$, denoted by $_{i}\mathscr{D}_{M}^{\alpha,\beta }$, is
\begin{equation}\label{J}
_{i}\mathscr{D}_{M}^{\alpha,\beta }f\left( t\right) :=\underset{\varepsilon \rightarrow 0}{%
\lim }\frac {f\left(t\;_{i}\mathbb{E}_{\beta }\left( \varepsilon t^{-\alpha }\right)
\right) -f\left( t\right) }{\varepsilon },
\end{equation}
$\forall t>0$ and $_{i}\mathbb{E}_{\beta }\left(\cdot\right) $, $\beta >0$ is a truncated Mittag-Leffler function of one parameter, as defined in \rm\text{Eq}.\rm(\ref{2A}).

Note that, if $f$ is $\alpha$-differentiable in some open interval $(0,a)$, $a>0$, and $\underset{t\rightarrow 0^{+}}{\lim }\left( _{i}\mathscr{D}_{M}^{\alpha,\beta }f\left(t\right) \right) $ exist, then we have
\begin{equation*}
_{i}\mathscr{D}_{M}^{\alpha,\beta }f\left( 0\right) =\underset{t\rightarrow 0^{+}}{\lim }\left( _{i}\mathscr{D}_{M}^{\alpha,\beta }f\left(t\right) \right).
\end{equation*}
\end{definition}

\begin{theorem} If a function $f:\left[ 0,\infty \right) \rightarrow \mathbb{R}$ is $\alpha$-differentiable for $t_{0}>0$, with $0<\alpha \leq 1$, $\beta>0$, then $f$ is continuous at $t_{0}$.
\end{theorem}
\begin{proof} In fact, let us consider the identity
\begin{equation}\label{B}
f\left(t_{0}\; _{i}\mathbb{E}_{\beta }\left( \varepsilon t_{0}^{-\alpha }\right) \right)
-f\left( t_{0}\right) =\left( \frac{f\left(t_{0}\; _{i}\mathbb{E}_{\beta }\left(
\varepsilon t_{0}^{-\alpha }\right) \right) -f\left( t_{0}\right) }{%
\varepsilon }\right) \varepsilon .
\end{equation}

Applying the limit $\varepsilon \rightarrow 0$ on both sides of {\rm\text{Eq}.\rm(\ref{B})}, we get
\begin{eqnarray*}
\underset{\varepsilon \rightarrow 0}{\lim }f\left( t_{0}\;_{i}\mathbb{E}_{\beta
}\left( \varepsilon t_{0}^{-\alpha }\right) \right) -f\left( t_{0}\right) 
&=&\underset{\varepsilon \rightarrow 0}{\lim }\left( \frac{f\left(t_{0}\;
_{i}\mathbb{E}_{\beta }\left( \varepsilon t_{0}^{-\alpha }\right) \right) -f\left(
t_{0}\right) }{\varepsilon }\right) \underset{\varepsilon \rightarrow 0}{%
\lim }\varepsilon   \notag \\
&=&_{i}\mathscr{D}_{M}^{\alpha,\beta }f\left( t\right) \underset{\varepsilon \rightarrow 0}{\lim }\varepsilon   \notag \\
&=&0.
\end{eqnarray*}

Then, $f$ is continuous at $t_{0}$.
\end{proof}

Using the definition of truncated Mittag-Leffler function of one parameter, we have
\begin{equation}\label{A}
f\left(t\;_{i}\mathbb{E}_{\beta }\left( \varepsilon t^{-\alpha }\right) \right) =f\left( t
\overset{i}{\underset{k=0}{\sum }}\frac{\left( \varepsilon t^{-\alpha
}\right) ^{k}}{\Gamma \left( \beta k+1\right) }\right) .
\end{equation}

Applying the limit $\varepsilon \rightarrow 0$ on both sides of Eq.(\ref{A}) and since $f$ is a continuous function, we have
\begin{eqnarray}
\underset{\varepsilon \rightarrow 0}{\lim }f\left(t\;_{i}\mathbb{E}_{\beta }\left(
\varepsilon t^{-\alpha }\right)\right)  &=&\underset{\varepsilon
\rightarrow 0}{\lim }f\left( t\overset{i}{\underset{k=0}{\sum }}%
\frac{\left( \varepsilon t^{-\alpha }\right) ^{k}}{\Gamma \left( \beta
k+1\right) }\right)   \notag \\
&=&f\left( t\underset{\varepsilon \rightarrow 0}{\lim }\overset{i}{%
\underset{k=0}{\sum }}\frac{\left( \varepsilon t^{-\alpha }\right) ^{k}}{%
\Gamma \left( \beta k+1\right) }\right) . 
\end{eqnarray}

Further, we have
\begin{eqnarray}\label{C1}
t\;_{i}\mathbb{E}_{\beta }\left( \varepsilon t^{-\alpha }\right)  &=&t\overset{i}{%
\underset{k=0}{\sum }}\frac{\left( \varepsilon t^{-\alpha }\right) ^{k}}{%
\Gamma \left( \beta k+1\right) }\notag \\
&=& t+\frac{\varepsilon t^{1-\alpha }}{\Gamma
\left( \beta +1\right) }+\frac{t\left( \varepsilon t^{-\alpha }\right) ^{2}}{%
\Gamma \left( 2\beta +1\right) }+\frac{t\left( \varepsilon t^{-\alpha
}\right) ^{3}}{\Gamma \left( 3\beta +1\right) }+\cdot \cdot \cdot +\frac{%
t\left( \varepsilon t^{-\alpha }\right) ^{i}}{\Gamma \left( i\beta +1\right) }.
\notag \\
\end{eqnarray}

Applying the limit $\varepsilon \rightarrow 0$ on both sides of \rm\text{Eq}.\rm(\ref{C1}), we have
\begin{equation*}
\underset{\varepsilon \rightarrow 0}{\lim }\overset{i}{\underset{%
k=0}{\sum }}\frac{\left( \varepsilon t^{-\alpha }\right) ^{k}}{\Gamma \left(
\beta k+1\right) }=1.
\end{equation*}
In this way, we conclude that
\begin{eqnarray}
\underset{\varepsilon \rightarrow 0}{\lim }f\left(t\;_{i}\mathbb{E}_{\beta }\left(
\varepsilon t^{-\alpha }\right) \right)  &=& f\left(t\right) . 
\end{eqnarray}

Here, we present the theorem that encompasses the main classical properties of integer order calculus. For the chain rule, it is verified through an example, as we will see next. We will do here, only the demonstration of the chain rule, for other items, follow the same steps of Theorem 2 found in the paper by Sousa and Oliveira \cite{JE}.

\begin{theorem}\label{JOSE} Let $0<\alpha \leq 1$, $\beta>0$, $a,b\in\mathbb{R}$ and $f, g$ $\alpha$-differentiable, at a point $t>0$. Then:
\begin{enumerate}

\item $ _{i}\mathscr{D}_{M}^{\alpha,\beta }\left( af+bg\right) \left( t\right) =a\; _{i}\mathscr{D}_{M}^{\alpha,\beta }f\left(
t\right) +b\; _{i}\mathscr{D}_{M}^{\alpha,\beta }g\left( t\right) $.

\item $ _{i}\mathscr{D}_{M}^{\alpha,\beta }\left( f\cdot g\right) \left( t\right) =f\left( t\right)\; _{i}\mathscr{D}_{M}^{\alpha,\beta }g\left( t\right) +g\left( t\right)\; _{i}\mathscr{D}_{M}^{\alpha,\beta }f\left( t\right)$.

\item $\displaystyle _{i}\mathscr{D}_{M}^{\alpha,\beta }\left( \frac{f}{g}\right) \left( t\right) =\frac{g\left(t\right)\; _{i}\mathscr{D}_{M}^{\alpha,\beta }f\left( t\right) -f\left( t\right) \;_{i}\mathscr{D}_{M}^{\alpha,\beta }g\left( t\right) }{\left[ g\left( t\right) \right] ^{2}} $.

\item $_{i}\mathscr{D}_{M}^{\alpha,\beta }\left( c\right) =0 $, where $f(t)=c$ is a constant.

  \item{\rm\text{(Chain rule)}}If $f$ is differentiable, then  $_{i}\mathscr{D}_{M}^{\alpha,\beta } \left( f\right)(t)= \displaystyle\frac{t^{1-\alpha}}{\Gamma \left(\beta +1\right) } \frac{df(t)}{dt}. $

\begin{proof} From {\rm\text{Eq}.\rm(\ref{C1})}, we have
\begin{eqnarray*}
t\;_{i}\mathbb{E}_{\beta }\left( \varepsilon t^{-\alpha }\right)  &=&t+\frac{\varepsilon t^{1-\alpha }}{\Gamma \left(\beta +1\right) }+O\left( \varepsilon^{2}\right),
\end{eqnarray*}
and introducing the following change,
\begin{equation*}
h=\varepsilon t^{1-\alpha }\left( \frac{1}{\Gamma \left(\beta +1\right) }+O\left( \varepsilon
\right) \right) \Rightarrow \varepsilon =\frac{h}{t^{1-\alpha }\left( \frac{1%
}{\Gamma \left(\beta +1\right) }+O\left( \varepsilon \right) \right) },
\end{equation*}
we conclude that
\begin{eqnarray*}
_{i}\mathscr{D}_{M}^{\alpha,\beta }f\left( t\right) &=&\underset{\varepsilon \rightarrow 0}{\lim }\frac{\displaystyle\frac{f\left(t+h\right) -f\left( t\right) }{ht^{\alpha -1}}}{\frac{1}{\Gamma \left(\beta +1\right) }\left(1+{\Gamma \left(\beta +1\right) } O\left( \varepsilon \right) \right) }  \notag \\
&=&\frac{t^{1-\alpha}}{\Gamma \left(\beta +1\right) }\underset{\varepsilon \rightarrow 0}{%
\lim }\frac{\displaystyle\frac{f\left( t+h\right) -f\left( t\right) }{h}}{1+{\Gamma \left(\beta +1\right) } O\left( \varepsilon \right) }  \notag \\
&=&\frac{t^{1-\alpha}}{\Gamma \left(\beta +1\right) }\frac{df\left( t\right) }{dt},
\end{eqnarray*}
with $\beta>0$ and $t>0$.
\end{proof}

\item $_{i}\mathscr{D}_{M}^{\alpha,\beta }\left(f\circ g\right)(t)=f'(g(t)) _{i}\mathscr{D}_{M}^{\alpha,\beta }g(t)$, for $f$ differentiable at $g(t)$.
\end{enumerate}
\end{theorem}

Now, it is necessary to know if, in addition to the previous Theorem 2 that contains important properties similar to integer-order calculus, this truncated $M$-fractional derivative type Eq.(\ref{J}) also has important theorems related to the classical calculus. We shall now see that: the Rolle's theorem, the mean value theorem and its extension coming from the integer-order calculus can be extended to $\alpha$-differentiable functions, i.e., that admit truncated $M$-fractional derivative as introduced in Eq.(\ref{J}).

\begin{theorem} {\rm\text{(Rolle's theorem for fractional $\alpha$-differentiable functions)}}
Let $a>0$, and $f:\left[ a,b\right] \rightarrow \mathbb{R}$ be a function with the properties:
\begin{enumerate}
\item $f$ is continuous on $[a,b]$.
\item $f$ is $\alpha$-differentiable on $(a,b)$ for some $\alpha\in(0,1)$.
\item  $f(a)=f(b)$.
\end{enumerate}
Then, $\exists c\in(a,b)$, such that $_{i}\mathscr{D}^{\alpha,\beta}_{M}f(c)=0$, with $\beta>0$.
\end{theorem}

\begin{proof}Since $f$ is continuous on $[a,b]$ and $f(a)=f(b)$, there exist $c\in(a,b)$, at which the function has a local extreme. Then,
\begin{eqnarray*}
_{i}\mathscr{D}_{M}^{\alpha,\beta }f\left( c\right)  &=&\underset{\varepsilon \rightarrow 0^{-}}{%
\lim }\frac{f\left(c\;_{i}\mathbb{E}_{\beta }\left( \varepsilon c^{-\alpha }\right)
\right) -f\left( c\right) }{\varepsilon }=\underset{\varepsilon \rightarrow 0^{+}}{\lim }\frac{f\left( c\;_{i}\mathbb{E}_{\beta }\left( \varepsilon c^{-\alpha }\right) \right) -f\left( c\right) }{%
\varepsilon }.
\end{eqnarray*}

But, the two limits have opposite signs. Hence, $_{i}\mathscr{D}_{M}^{\alpha,\beta }f\left( c\right) =0$.
\end{proof}

The proof of Theorem 4 and Theorem 5, will be omitted, but follow the same reasoning of the respective theorems demonstrated in Sousa and Oliveira \cite{JE}.

\begin{theorem} {\rm\text{(Mean-value theorem for fractional $\alpha$-differentiable functions)}}
Let $a>0$ and $f:\left[ a,b\right] \rightarrow \mathbb{R}$ be a function with the properties:
\begin{enumerate}
\item $f$ is continuous on $[a,b]$.
\item $f$ is $\alpha$-differentiable on $(a,b)$ for some $\alpha\in(0,1)$.
\end{enumerate}
Then, $\exists c\in(a,b)$, such that
\begin{equation*}
_{i}\mathscr{D}_{M}^{\alpha,\beta }f\left( c\right) =\frac{f\left( b\right) -f\left(
a\right) }{\displaystyle\frac{b^{\alpha }}{\alpha }-\frac{a^{\alpha }}{\alpha }},
\end{equation*}
with $\beta>0$.
\end{theorem}

\begin{theorem}{\rm\text{(Extension mean value theorem for fractional $\alpha$-differentiable functions)}} Let $a>0$, and $f,g:\left[ a,b\right] \rightarrow \mathbb{R}$ functions that satisfy:
\begin{enumerate}
\item $f, g$ are continuous on $[a,b]$.
\item $f, g$ are $\alpha$-differentiable for some $\alpha\in(0,1)$.
\end{enumerate}
Then, $\exists c\in(a,b)$, such that
\begin{equation}
\frac{_{i}\mathscr{D}_{M}^{\alpha,\beta }f\left( c\right) }{_{i}\mathscr{D}_{M}^{\alpha,\beta }g\left( c\right) }=\frac{f\left( b\right) -f\left( a\right) }{g\left(
b\right) -g\left( a\right) },
\end{equation}
with $\beta>0$.
\end{theorem}

\begin{definition} Let $\alpha\in(n,n+1]$, for some $n\in\mathbb{N}$, $\beta>0$ and $f$ $n$-differentiable for $t>0$. Then the $\alpha$-fractional derivative of $f$ is defined by
\begin{equation}
_{i}\mathscr{D}_{M}^{\alpha,\beta;n }f\left( t\right) :=\underset{\varepsilon \rightarrow 0}{%
\lim }\frac {f^{(n)}\left(t\; _{i}\mathbb{E}_{\beta }\left( \varepsilon t^{n-\alpha }\right)
\right) -f^{(n)}\left( t\right) }{\varepsilon },
\end{equation}
since the limit exist.
\end{definition}

From Definition 2 and the chain rule, that is, from item 5 of Theorem 2, by induction on $n$, we can prove that $_{i}\mathscr{D}_{M}^{\alpha,\beta;n }f\left( t\right)= \displaystyle\frac{t^{n+1-\alpha}}{\Gamma \left(\beta +1\right) } f^{(n+1)}(t)$, $\alpha\in(n,n+1]$ and $f$ is $(n+1)$-differentiable for $t>0$.

Now, we know that: this truncated $M$-fractional derivative type Eq.(\ref{J}) has a corresponding $M$-fractional integral. Then, we will present the definition and a theorem that corresponds to the inverse property. For other results involving integrals, one can consult \cite{JE,OSEN}.

\begin{definition} Let $a\geq 0$ and $t\geq a$. Also, let $f$ be a function defined in $(a,t]$ and $0<\alpha <1$. Then, the $M$-fractional integral of order $\alpha$ of function $f$ is defined by \rm\cite{JE}
\begin{equation}\label{ZA}
_{M}\mathcal{I}_{a}^{\alpha,\beta }f\left( t\right) ={\Gamma \left(\beta +1\right) }\int_{a}^{t}\frac{f\left( x\right) }{x^{1-\alpha }}dx,
\end{equation}
with $\beta>0$.
\end{definition}

\begin{theorem} \rm\text{(Inverse)} Let $a\geq 0$ and $0<\alpha < 1$. Also, let $f$ be a continuous function such that exist $_{M}\mathcal{I}_{a}^{\alpha,\beta }f$. Then
\begin{equation}
_{i}\mathscr{D}_{M}^{\alpha,\beta }(_{M}\mathcal{I}_{a}^{\alpha,\beta }f\left( t\right)) =f(t),
\end{equation}
with  $t\geq a$ and $\beta>0$.
\end{theorem}

\begin{proof} In fact, using the chain rule as seen in \rm\text{Theorem} \rm\ref{JOSE}, we have
\begin{eqnarray}\label{Z1}
_{i}\mathscr{D}_{M}^{\alpha,\beta }\left( _{M}\mathcal{I}_{a}^{\alpha,\beta }f\left(
t\right) \right)  &=&\frac{t^{1-\alpha }}{\Gamma \left( \beta +1\right) }\frac{d}{dt}(_{M}\mathcal{I}_{a}^{\alpha,\beta }f\left( t\right))   \notag \\
&=&\frac{t^{1-\alpha }}{\Gamma \left( \beta +1\right) }\frac{d}{dt}\left( {\Gamma \left( \beta +1\right) }\int_{a}^{t}\frac{f\left( x\right) }{x^{1-\alpha }}dx\right)   \notag \\
&=&\frac{t^{1-\alpha }}{\Gamma \left( \beta +1\right) }\left( \frac{\Gamma
\left( \beta +1\right) }{t^{1-\alpha }}f\left( t\right) \right)   \notag \\
&=&f\left( t\right) .
\end{eqnarray}
\end{proof}

With the condition $f(a)=0$, by Theorem 6, that is, Eq.(\ref{Z1}), we have $_{M}\mathcal{I}_{a}^{\alpha,\beta } \left[ {_{i}\mathscr{D}_{M}^{\alpha,\beta }}f(t)\right]=f(t)$.


\section{Relation with other fractional derivatives types}

In this section, we will discuss the relationship between the fractional conformable derivative proposed by Khalil et al. \cite{RMAM}, the alternative fractional derivative and the generalized alternative fractional derivative proposed by Katugampola \cite{UNK} and the $M$-fractional derivative proposed by Sousa and Oliveira \cite{JE}, with our truncated $M$-fractional derivative type.

Khalil et al. \cite{RMAM} proposed a definition of a fractional derivative, called conformable fractional derivative that refers to the classical properties of integer order calculus, given by
\begin{equation}\label{K1}
f^{(\alpha) }\left( t\right) =\underset{\varepsilon \rightarrow 0}{\lim }\frac{f\left( t+\varepsilon t^{1-\alpha}\right) -f\left( t\right) }{\varepsilon },
\end{equation}
with $\alpha\in(0,1)$ and $t>0$.

In 2014, Katugampola \cite{UNK} proposed another definition of a fractional derivative, called an alternative fractional derivative which also refers to the classical properties of integer-order calculus, given by
\begin{equation}\label{K2}
\mathcal{D}^{\alpha }f\left( t\right) =\underset{\varepsilon \rightarrow 0}{\lim }\frac{f\left( te^{\varepsilon t^{-\alpha }}\right) -f\left( t\right) }{\varepsilon },
\end{equation}
with $\alpha\in(0,1)$ and $t>0$.

In the same paper, Katugampola \cite{UNK} by means of a truncated exponential function, that is, $_{k}e^{x}$, proposed another generalized fractional derivative, given by
\begin{equation}\label{K3}
\mathcal{D}^{\alpha }_{k}f\left( t\right) =\underset{\varepsilon \rightarrow 0}{\lim }\frac{f\left(_{k}e^{\varepsilon t^{-\alpha }}t\right) -f\left( t\right) }{\varepsilon },
\end{equation}
with $\alpha\in(0,1)$ and $t>0$.

Recently, Sousa and Oliveira \cite{JE} introduced the $M$-fractional derivative $\mathscr{D}_{M}^{\alpha,\beta}$ where the parameter $\beta>0$ and $M$ is the notation to designate that the function to be derived involves the Mittag-Leffler function of one parameter, given by

\begin{equation}\label{K4}
\mathscr{D}_{M}^{\alpha,\beta }f\left( t\right) :=\underset{\varepsilon \rightarrow 0}{%
\lim }\frac {f\left( t\;\mathbb{E}_{\beta }\left( \varepsilon t^{-\alpha }\right)
\right) -f\left( t\right) }{\varepsilon },
\end{equation}
with $\alpha\in(0,1)$ and $t>0$.

It is clear that our definition of truncated $M$-fractional derivative type Eq.(\ref{J}) is more general than the fractional derivatives Eq.(\ref{K1}), Eq.(\ref{K2}), Eq.(\ref{K3}) and Eq.(\ref{K4}). We will now study particular cases involving such fractional derivatives.

Choosing $\beta=1$ and applying the limit $i\rightarrow 0$ on both sides of Eq.(\ref{J}), we have
\begin{equation}
_{1}\mathscr{D}_{M}^{\alpha,\beta }f\left( t\right)=\underset{\varepsilon \rightarrow 0}{%
\lim }\frac {f\left(t\;_{1}\mathbb{E}_{1}\left( \varepsilon t^{-\alpha }\right)
\right) -f\left( t\right) }{\varepsilon }.
\end{equation}

But, it is know that
\begin{equation}
_{1}\mathbb{E}_{1}\left( \varepsilon t^{-\alpha }\right) =\overset{1}{%
\underset{k=0}{\sum }}\frac{\left( \varepsilon t^{-\alpha }\right) ^{k}}{%
\Gamma \left( k+1\right) }=1+\varepsilon t^{-\alpha }.
\end{equation}

Thus, we conclude that
\begin{equation}
_{1}\mathscr{D}_{M}^{\alpha,\beta }f\left( t\right)=\underset{\varepsilon \rightarrow 0}{%
\lim }\frac {f\left(t+\varepsilon t^{1-\alpha }\right) -f\left( t\right) }{\varepsilon }=f^{(\alpha)}(t),
\end{equation}
which is exactly the conformable fractional derivative Eq.$(\ref{K1})$.

Choosing $\beta=1$ and applying the limit $i\rightarrow \infty$ on both sides of Eq.(\ref{J}), we have
\begin{equation}
_{\infty}\mathscr{D}_{M}^{\alpha,\beta }f\left( t\right)=\underset{\varepsilon \rightarrow 0}{%
\lim }\frac {f\left(t\;_{\infty}\mathbb{E}_{1}\left( \varepsilon t^{-\alpha }\right)
\right) -f\left( t\right) }{\varepsilon }.
\end{equation}

But, as we have
\begin{equation}
_{\infty}\mathbb{E}_{1}\left( \varepsilon t^{-\alpha }\right) =\overset{\infty}{%
\underset{k=0}{\sum }}\frac{\left( \varepsilon t^{-\alpha }\right) ^{k}}{%
\Gamma \left( k+1\right) }=e^{\varepsilon t^{-\alpha }},
\end{equation}
we conclude that
\begin{equation}
_{\infty}\mathscr{D}_{M}^{\alpha,\beta }f\left( t\right)=\underset{\varepsilon \rightarrow 0}{%
\lim }\frac {f\left(t e^{\varepsilon t^{-\alpha }}\right) -f\left( t\right) }{\varepsilon }=\mathcal{D}^{\alpha }f\left( t\right),
\end{equation}
which is exactly the alternative fractional derivative, Eq.$(\ref{K2})$.

Choosing $\beta=1$ in Eq.(\ref{J}), we have
\begin{equation}
_{i}\mathscr{D}_{M}^{\alpha,\beta }f\left( t\right)=\underset{\varepsilon \rightarrow 0}{%
\lim }\frac {f\left(t\;_{i}\mathbb{E}_{1}\left( \varepsilon t^{-\alpha }\right)
\right) -f\left( t\right) }{\varepsilon }.
\end{equation}

Remembering that
\begin{equation}
_{i}\mathbb{E}_{1}\left( \varepsilon t^{-\alpha }\right) =\overset{i}{%
\underset{k=0}{\sum }}\frac{\left( \varepsilon t^{-\alpha }\right) ^{k}}{%
\Gamma \left( k+1\right) }=e_{i}^{\varepsilon t^{-\alpha }},
\end{equation}
we have
\begin{equation}
_{i}\mathscr{D}_{M}^{\alpha,\beta }f\left( t\right)=\underset{\varepsilon \rightarrow 0}{%
\lim }\frac {f\left(e_{i}^{\varepsilon t^{-\alpha }}t\right) -f\left( t\right) }{\varepsilon }=\mathcal{D}_{i}^{\alpha }f\left( t\right),
\end{equation}
exactly the generalized fractional derivative, Eq.$(\ref{K3})$.

Finally, applying the limit $i\rightarrow \infty$ on both sides of Eq.(\ref{J}), we have
\begin{equation}
_{\infty}\mathscr{D}_{M}^{\alpha,\beta }f\left( t\right)=\underset{\varepsilon \rightarrow 0}{%
\lim }\frac {f\left(t\;_{\infty}\mathbb{E}_{\beta}\left( \varepsilon t^{-\alpha }\right)
\right) -f\left( t\right) }{\varepsilon },
\end{equation}
since
\begin{equation}
_{\infty}\mathbb{E}_{\beta}\left( \varepsilon t^{-\alpha }\right) =\overset{\infty}{%
\underset{k=0}{\sum }}\frac{\left( \varepsilon t^{-\alpha }\right) ^{k}}{%
\Gamma \left( k+1\right) }=\mathbb{E}_{\beta}\left( \varepsilon t^{-\alpha }\right),
\end{equation}
we conclude that
\begin{equation}
_{\infty}\mathscr{D}_{M}^{\alpha,\beta }f\left( t\right)=\underset{\varepsilon \rightarrow 0}{%
\lim }\frac {f\left(t\;\mathbb{E}_{\beta}\left( \varepsilon t^{-\alpha }\right)\right) -f\left( t\right) }{\varepsilon }=\mathscr{D}_{M}^{\alpha,\beta }f\left( t\right),
\end{equation}
exactly the $M$-fractional derivative, Eq.$(\ref{K4})$.


\section{Application}

In this section, we obtain the solution of the heat equation using a truncated $M$-fractional derivative type with $0<\alpha <1 $ and present some graphs about the behavior of the solution.

Consider the heat equation in one dimension given by
\begin{equation*}
\frac{\partial u\left( x,t\right) }{\partial t}=k\frac{\partial ^{2}u\left(
x,t\right) }{\partial x^{2}},\text{ }0<x<L,\text{ }t>0,
\end{equation*}
where $k$ is a positive constant.

Using a $M$-fractional derivative type, we propose an $M$-fractional heat equation given by
\begin{equation}\label{B}
\frac{\partial ^{\alpha }u\left( x,t\right) }{\partial t^{\alpha }}=k\frac{%
\partial ^{2}u\left( x,t\right) }{\partial x^{2}},\text{ }0<x<L,\text{ }t>0,
\end{equation}
where $0<\alpha <1$ and with the initial condition and boundary conditions given by
\begin{eqnarray}\label{1}
u\left( 0,t\right)  &=&0\,, t\geq 0, \\\notag
u\left( L,t\right)  &=&0\,, t\geq 0, \\\notag
u\left( x,0\right)  &=&f\left( x\right) \,, 0\leq x\leq L.\notag
\end{eqnarray}

We start, considering the so-called $M$-fractional linear differential equation with constant coefficients
\begin{equation}\label{C}
\frac{\partial ^{\alpha }v\left( x,t\right) }{\partial t^{\alpha }}\pm \mu
^{2}v\left( x,t\right) =0,
\end{equation}
where $\mu^{2}$ is a positive constant.

Using the item 5 in Theorem \ref{JOSE}, the {\rm Eq.(\ref{B})} can be written as follows
\begin{equation*}
\frac{t^{1-\alpha }}{\Gamma \left( \beta +1\right) }\frac{dv\left(
x,t\right) }{dt}\pm \mu ^{2}v\left( x,t\right) =0,
\end{equation*}
whose solution is given by
\begin{equation}\label{D}
v\left( t\right) =ce^{\pm \Gamma \left( \beta +1\right) \frac{\mu
^{2}t^{\alpha }}{\alpha }},
\end{equation}
with $0<\alpha<1$ and $\beta>0$.

Now, we will use separation of variables method to obtain the solution of the $M$-fractional heat equation. Then, considering $u\left( x,t\right) =P\left( x\right) Q\left( t\right) $ and replacing in Eq.(\ref{B}), we get
\begin{equation*}
\frac{d^{\alpha }}{dt^{\alpha }}Q\left( t\right) P\left( x\right) =k\frac{%
d^{2}}{dx^{2}}P\left( x\right) Q\left( t\right) 
\end{equation*}
which implies
\begin{equation}\label{E}
\frac{1}{kQ\left( t\right) }\frac{d^{\alpha }}{dt^{\alpha }}Q\left( t\right)
=\frac{1}{P\left( x\right) }\frac{d^{2}}{dx^{2}}P\left( x\right) =\xi .
\end{equation}

From Eq.(\ref{E}), we obtain a system of differential equations, given by
\begin{equation}\label{F}
\frac{d^{\alpha }}{dt^{\alpha }}Q\left( t\right) -k\xi Q\left( t\right) =0
\end{equation}
and
\begin{equation}\label{G}
\frac{d^{2}}{dx^{2}}P\left( x\right) -\xi P\left( x\right) =0.
\end{equation}

First, let's find the solution of Eq.(\ref{G}). For this, we must study three cases, that is, $\xi =0,$ $\xi =-\mu ^{2}$ e $\xi =\mu ^{2}.$

Case 1: $\xi =0$.

Substituting $ \xi = 0 $ into Eq.(\ref{G}), we have
\begin{equation}
\frac{d^{2}}{dx^{2}}P\left( x\right) -\xi P\left( x\right) =0,
\end{equation}
whose solution is given by $P\left( x\right) =c_{1}x+c_{2}$, with $c_{1}$ and $c_{2}$ arbitrary constant. Using the initial conditions given by Eq.(\ref{1}), we obtain that $ c_{1}=c_{2}=0$. Like this, $P\left( x\right) =0$, which implies $u(x,t)=0$ trivial solution.

Case 2: $\xi =-\mu ^{2}$.

Substituting $\xi =-\mu ^{2}$ into Eq.(\ref{G}), we get
\begin{equation*}
\frac{d^{2}}{dx^{2}}P\left( x\right) +\mu ^{2}P\left( x\right) =0,
\end{equation*}
whose solution is given by $P\left( x\right) =c_{2}\sin \left( \mu x\right) +c_{1}\cos \left( \mu x\right)$, with $c_{1}$ and $c_{2}$ arbitrary constant. Using the initial conditions Eq.(\ref{1}), we obtain $c_{1}=0$ and $0=c_{2}\sin \left( \mu x\right)$ which implies that $\mu =\frac{n\pi }{L},$ with $n=1,2,...$. Then, we obtain
\begin{equation*}
P_{n}\left( x\right) =a_{n}\sin \left( \frac{n\pi x}{L}\right) \text{ and } \mu =\frac{n\pi }{L}.
\end{equation*}

Case 3: $\xi =\mu ^{2}$.

Substituting $\xi =\mu ^{2}$ into Eq.(\ref{G}), we get
\begin{equation*}
\frac{d^{2}}{dx^{2}}P\left( x\right) -\mu ^{2}P\left( x\right) =0
\end{equation*}
whose solution is given by $P\left( x\right) =c_{1}$ $e^{\mu x}+c_{2}$ $e^{-\mu x}=A\cosh \left( \mu x\right) +B\sinh \left( \mu x\right)$, with $c_{1}$, $c_{2}$, $A$, $B$ arbitrary constant. Using the boundary conditions Eq.(\ref{1}), we have $A=0$ and $0=B\sinh \left( \mu x\right)$. As $\lambda =-\mu ^{2}<0$ and $\lambda L\neq 0$ then $\sinh \left( \mu x\right) \neq 0.$ Like this, we get $B=0$ and then $P_{n}\left( x\right) =0$, which implies $u(x,t)=0$, trivial solution.

Therefore, the solution of Eq.(\ref{G}) is given by
\begin{equation}\label{H}
P_{n}\left( x\right) =a_{n}\sin \left( \frac{%
n\pi x}{L}\right) \text{ and} \mu =\frac{n\pi }{L}.
\end{equation}

Using the Eq.( \ref{C}) and Eq.(\ref{D}), we have
\begin{equation}\label{I}
Q_{n}\left( t\right) =b_{n}\exp \left( -\Gamma \left( \beta +1\right) \left( 
\frac{n\pi }{L}\right) ^{2}\frac{k}{\alpha }t^{\alpha }\right),
\end{equation}
where $b_{n}$ are constant coefficients.

So, using the Eq.(\ref{H}) and Eq.(\ref{I}), the partial solutions of Eq.(\ref{B}), is given by
\begin{equation*}
u_{\beta}\left( x,t\right) =\underset{n=1}{\overset{\infty }{\sum }}c_{n}\sin \left( 
\frac{n\pi x}{L}\right) \exp \left( -\Gamma \left( \beta +1\right) \left( 
\frac{n\pi }{L}\right) ^{2}\frac{k}{\alpha }t^{\alpha }\right) .
\end{equation*}

Using Eq.(\ref{1}), we get
\begin{equation*}
u\left( x,0\right) =f\left( x\right) =\underset{n=1}{\overset{\infty }{\sum }%
}c_{n}\sin \left( \frac{n\pi x}{L}\right) 
\end{equation*}
which provides $c_{n}$ through
\begin{equation*}
c_{n}=\frac{2}{L}\int_{0}^{L}f\left( x\right) \sin \left( \frac{n\pi x}{L}%
\right) dx.
\end{equation*}

So, we conclude that the solution of $M$-fractional heat equation Eq.(\ref{B}), satisfying the conditions Eq.(\ref{1}), is given by
\begin{equation}\label{fig1}
u_{\beta}\left( x,t\right) =\underset{n=1}{\overset{\infty }{\sum }}\sin \left( 
\frac{n\pi x}{L}\right) \exp \left( -\Gamma \left( \beta +1\right) \left( 
\frac{n\pi }{L}\right) ^{2}\frac{k}{\alpha }t^{\alpha }\right) \left( \frac{2%
}{L}\int_{0}^{L}f\left( x\right) \sin \left( \frac{n\pi x}{L}\right)
dx\right) .
\end{equation}

Taking the limit $\beta \rightarrow 1$ in the last equation and using Eq.(\ref{J}), we have
\begin{equation}\label{fig2}
u\left( x,t\right) =\underset{n=1}{\overset{\infty }{\sum }}\sin \left( 
\frac{n\pi x}{L}\right) \exp \left( -\left( \frac{n\pi }{L}\right) ^{2}\frac{%
k}{\alpha }t^{\alpha }\right) \left( \frac{2}{L}\int_{0}^{L}f\left( x\right)
\sin \left( \frac{n\pi x}{L}\right) dx\right) ,
\end{equation}
which is exactly the solution of the fractional heat equation proposed by Çenesiz et al. \cite{YCA}.

On the other hand, taking the limit $\beta \rightarrow 1$ and $\alpha \rightarrow 1$, using Eq.(\ref{J}), we recover the solution of heat equation of integer order.
\begin{equation}\label{fig3}
u\left( x,t\right) =\underset{n=1}{\overset{\infty }{\sum }}\sin \left( 
\frac{n\pi x}{L}\right) \exp \left( -\left( \frac{n\pi }{L}\right)
^{2}kt\right) \left( \frac{2}{L}\int_{0}^{L}f\left( x\right) \sin \left( 
\frac{n\pi x}{L}\right) dx\right) .
\end{equation}

Next, we will present some plots by choosing values for the parameters $\beta$ and $\alpha$, to see the behavior of the solution presented in Eq (\ref{J}). The graphics were plotted using MATLAB 7:10 software (R2010a). For the elaboration of the following plots, we choose the function $f(x)=50x(1-x)$ and for each fixed $\beta$, we vary the $\alpha$ parameter.

\begin{figure}[h!]
\caption{Analytical solution of the $M$-fractional heat equation {\rm Eq.(\ref{fig1})}. We consider the values $ \beta = 0.5 $, $L = 1$, $k = 0.003$ and at time $t = 150$.}
\centering 
\includegraphics[width=14.0cm]{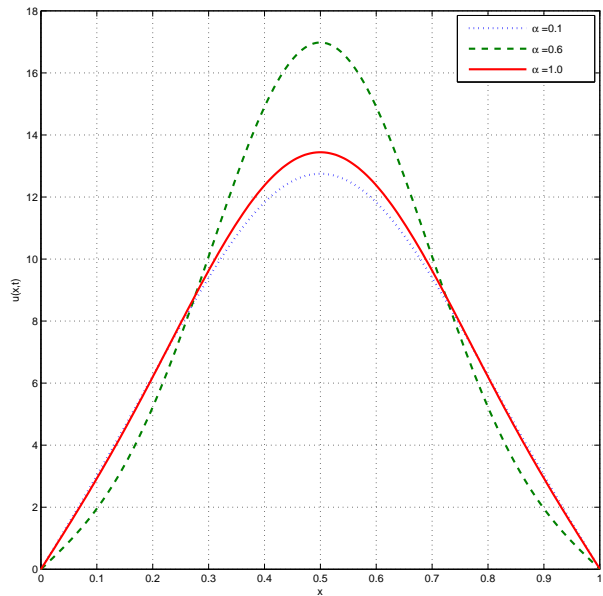} 
\label{fig:eryt1}
\end{figure}

\newpage

\begin{figure}[h!]
\caption{Analytical solution of the $M$-fractional heat equation {\rm Eq.(\ref{fig1})}. We take the values $ \beta = 1.0 $, $L = 1$, $k = 0.003$ and at time $t = 150$.}
\centering 
\includegraphics[width=14.0cm]{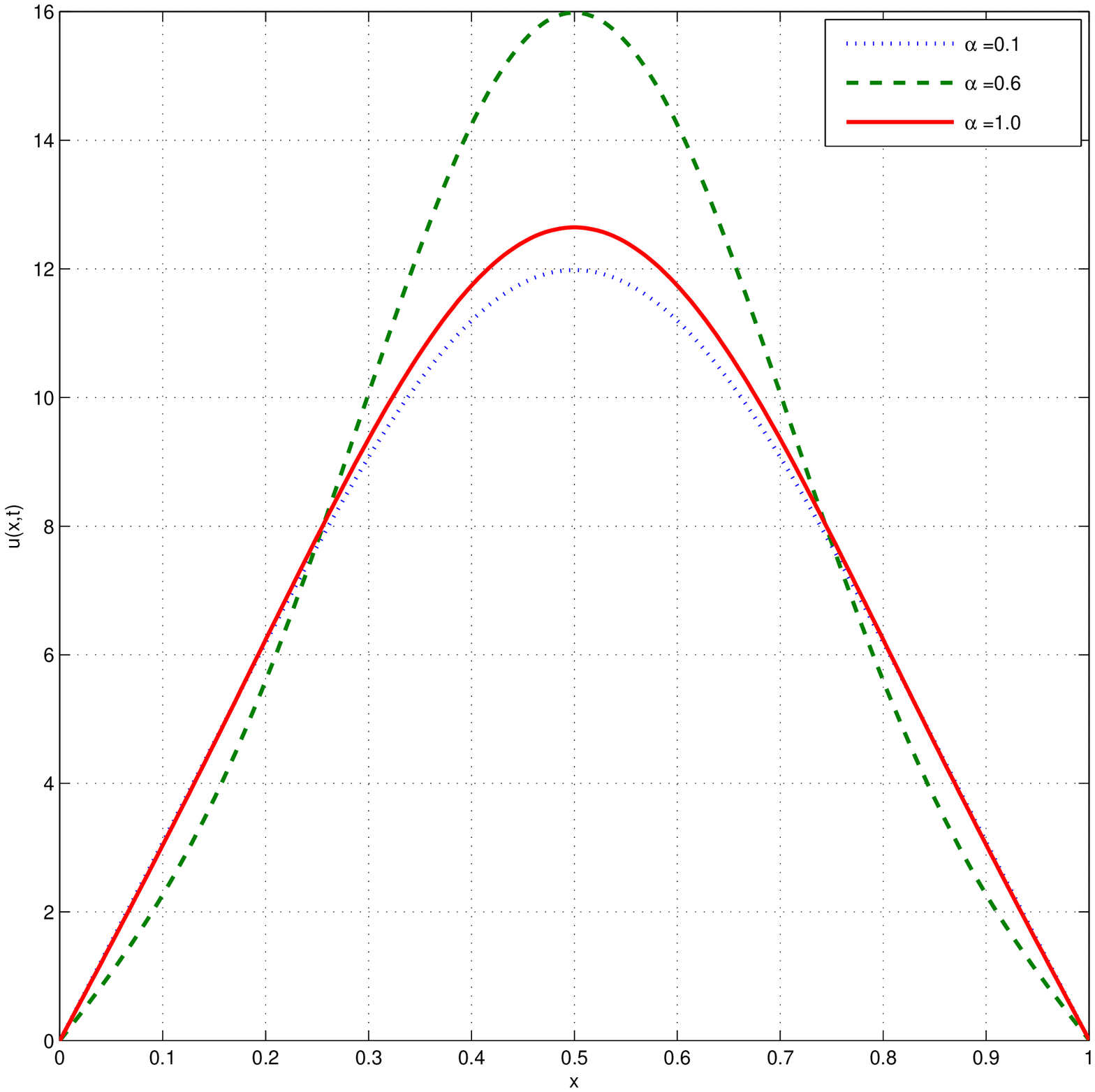} 
\label{fig:eryt1}
\end{figure}

\begin{figure}[h!]
\caption{Analytical solution of the $M$-fractional heat equation {\rm Eq.(\ref{fig1})}. We chose the values $ \beta = 2.0 $, $L = 1$, $k = 0.003$ and at time $t = 150$.}
\centering 
\includegraphics[width=14.0cm]{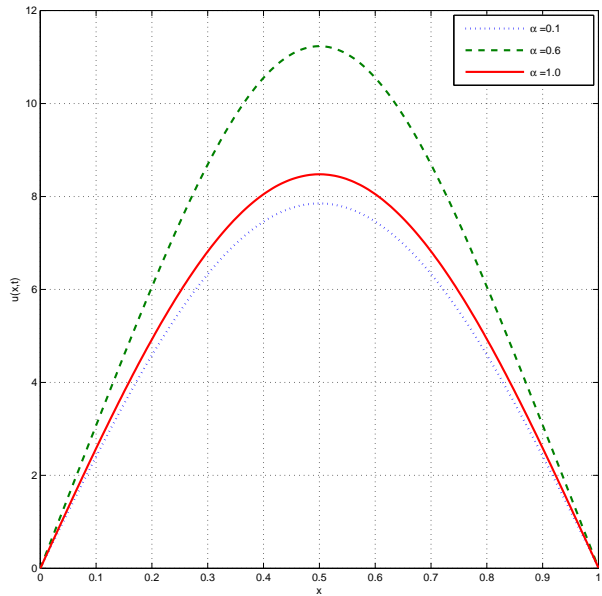} 
\label{fig:eryt1}
\end{figure}

\newpage

\section{Concluding remarks} 
We introduced a new truncated $M$-fractional derivative type for $\alpha$-differentiable functions and consequently its $M$-fractional integral, we obtained important results with respect to the properties of the integer-order derivative.

For a class of $\alpha$-differentiable functions, in the context of fractional derivatives, we conclude that this truncated $M$-fractional derivative type proposed here behaves well with respect to the classical properties of integer-order calculus. 

Using the truncated Mittag-Leffler function, it was possible to introduce a truncated $M$-fractional derivative type associated with $\alpha$-differentiable functions and consequently we obtained a very important relation with the fractional derivatives mentioned in the paper as seen in section 3. We conclude that the presented results contain as particular cases the derivatives proposed by Khalil et al. \cite {RMAM}, Katugampola \cite{UNK} and Sousa and Oliveira \cite{JE}. In addition, using our truncated $M$-fractional derivative type we performed and discussed the analytical solution of the heat equation.

\bibliography{ref}
\bibliographystyle{plain}

\end{document}